\documentclass[10pt]{amsart}
\usepackage{amssymb}
\usepackage{graphicx, pinlabel}
\usepackage{graphics}
\usepackage{amscd}
\usepackage{enumerate}
\usepackage{comment}
\usepackage[all]{xy}
\usepackage{mathtools} 
\usepackage{color} 
\usepackage{xspace} 
\usepackage[margin=3cm, marginparwidth=2.5cm]{geometry}

\newtheorem{thm}{Theorem}[section]
\newtheorem{theorem}[thm]{Theorem}

\newtheorem{corollary}[thm]{Corollary}

\newtheorem{lemma}[thm]{Lemma}

\newtheorem{proposition}[thm]{Proposition}

\theoremstyle{definition}

\newtheorem{definition}[thm]{Definition}

\newtheorem{question}[thm]{Question}
\theoremstyle{remark}

\newtheorem{remark}[thm]{Remark}

\newtheorem{example}[thm]{Example}

\numberwithin{equation}{section}

\newcommand{\ol}{\overline}
\newcommand{\td}{\widetilde}
\newcommand{\mb}{\mathbb}

\newcommand{\vsimeq}{\rotatebox[origin=c]{-90}{\footnotesize $\backsimeq$}}

\renewcommand{\theta}{\vartheta}

\renewcommand{\ni}{\nu}

\DeclarePairedDelimiter{\floor}{\lfloor}{\rfloor}

\newcommand{\set}[1]{\left\{#1\right\}}
\newcommand{\spn}[1]{\langle#1\rangle}
\newcommand{\scal}[2]{{\langle #1, #2 \rangle}}

\DeclareMathOperator{\im}{im}

\DeclareMathOperator{\lk}{lk}
\newcommand{\de}{\partial}
\newcommand{\sm}{\setminus}
\newcommand{\N}{\mathbb{N}}
\newcommand{\Z}{\mathbb{Z}}
\newcommand{\Q}{\mathbb{Q}}

\newcommand{\F}{\mathbb{F}}
\newcommand{\RP}{\mathbb{RP}}
\newcommand{\nbd}{\mathcal{N}} 
\newcommand{\Wo}{W^\circ}

\newcommand{\Fh}{\widehat{F}}
\newcommand{\mK}{\overline{K}}
\newcommand{\A}{\mathcal{A}}
\newcommand{\G}{\Gamma}

\DeclareMathOperator{\HF}{HF}
\newcommand{\HFi}{\HF^\infty}

\DeclareMathOperator{\CFK}{CFK}
\newcommand{\CFKi}{\CFK^\infty}

\newcommand{\CFKm}{\CFK^-}

\newcommand{\rCFKm}{\underline{\CFK}^-}

\renewcommand{\a}{\alpha}
\renewcommand{\b}{\beta}
\newcommand{\g}{\gamma}

\newcommand{\ours}{\varphi} 
\newcommand{\ourlim}{\omega}

\DeclareMathOperator{\Spin}{Spin}
\DeclareMathOperator{\tor}{tor}
\newcommand{\spinc}{spin$^c$\xspace}
\newcommand{\Spinc}{\Spin^c}

\newcommand{\Spinctor}{\Spin^c_{\tor}}
\newcommand{\s}{\mathfrak{s}}
\newcommand{\rs}{\s^\circ}

\renewcommand{\t}{\mathfrak{t}}

\begin{document}

\title[Correction terms and the non-orientable genus]{Correction terms and\\the non-orientable slice genus}

\author{Marco Golla}
\address{Mathematical Institute, University of Uppsala,
Box 480, 751 06 Uppsala, Sweden}
\email{marco.golla@math.uu.se}
\thanks{MG is supported by the Alice and Knut Wallenberg foundation.}

\author{Marco Marengon}
\address{Department of Mathematics, Imperial College London,
180 Queen's Gate, London SW7 2AZ, UK}
\email{m.marengon13@imperial.ac.uk}
\thanks{MM was supported by an EPSRC Doctoral Training Award.}

\date{}

\begin{abstract}
By considering negative surgeries on a knot $K$ in $S^3$, we derive a lower bound to the non-orientable slice genus $\g_4(K)$
in terms of the signature $\sigma(K)$ and the concordance invariants $V_i(\mK)$, which strengthens a previous bound given by Batson,
and which coincides with Ozsv\'ath--Stipsicz--Szab\'o's bound in terms of their $\upsilon$ invariant for L-space knots and quasi-alternating knots.
A curious feature of our bound is superadditivity, implying, for instance, that the bound on the stable non-orientable genus is sometimes better than the one on $\g_4(K)$.
\end{abstract}

\maketitle

\section{Introduction}
\label{sec:introduction}
Given a knot $K$ in $S^3$, it is a very classical problem to determine the minimal genus of an orientable surface $F$ in $B^4$ whose boundary is $K$. More recently, some attention has been drawn to the case of non-orientable surfaces instead. Namely, one can define $\gamma_4(K)$ as the minimal non-orientable genus among all such surfaces, where the non-orientable genus of $F$ is defined as $b_1(F)$.

Batson, and Ozsv\'ath, Stipsicz, and Szab\'o, on the other hand, gave lower bounds in terms of Heegaard Floer data.
More precisely, Batson proved that
\begin{equation}\label{e:batson}
\gamma_4(K) \ge \frac{\sigma(K)}2 - d(S^3_{-1}(K)),
\end{equation}
where $d(S^3_{-1}(K))$ is the Heegaard Floer correction term (or $d$-invariant) of the 3-manifold obtained as $(-1)$-surgery along $K$, in its unique \spinc structure (which is hence omitted from the notation)~\cite{Batson}.
Ozsv\'ath, Szab\'o and Stipsicz proved that
\begin{equation}\label{e:OSSz}
\gamma_4(K) \ge \left|\frac{\sigma(K)}2 - \upsilon(K)\right|,
\end{equation}
where $\upsilon$ is a concordance invariant defined in terms of the Floer homology package~\cite[Theorem 1.2]{OSSz-unoriented}.
Gilmer and Livingston gave lower bounds on $\gamma_4$ using Casson--Gordon invariants~\cite{GilmerLivingston}.

The main goal of this manuscript is to provide a new lower bound that generalises Batson's. It will be phrased in terms of the concordance invariants $\{V_i(\overline K)\}_i$ associated to the mirror $\overline K$ of $K$; these invariants were defined by Rasmussen~\cite{rasmussenVi} and further studied by Ni and Wu~\cite{NiWu} (see also Section~\ref{sec:dinvariants} below).
We will further package these invariants into a single integer-valued invariant that we call $\ours$, $\ours(K) = \min_{m\ge 0}\{m+2V_m(\overline K)\}$.

\begin{thm}\label{t:main}
For every knot $K$ in $S^3$,
\begin{equation}\label{e:us}
\gamma_4(K) \ge \frac{\sigma(K)}2 - \ours(K).
\end{equation}
\end{thm}
The existence of such a bound was indicated, but not made explicit, by Batson in his PhD thesis~\cite{Batson-PhD}. Moreover, since $d(S^3_{-1}(K)) = 2V_0(\overline K)\ge\ours(K)$, this is a strengthening of~\eqref{e:batson}.
Equation~\eqref{e:us} also implies the existence of a bound in terms of the invariant $\ni^+$ defined by Hom and Wu~\cite{HomWu}. By definition, one has $V_{\ni^+(\mK)}(\mK)=0$, so Theorem~\ref{t:main} implies at once
\[
\g_4(K) \geq \frac{\sigma(K)}{2} - \ni^+(\mK).
\]
Note that this bound is formally identical to~\eqref{e:batson},~\eqref{e:OSSz} and~\eqref{e:us}; to the best of the authors' knowledge, this bound never appeared in literature.

We will show below that the bound of Theorem~\ref{t:main} is sharp (see Remark~\ref{rem:sharpness}), and agrees with the one of~\eqref{e:OSSz} in the case of alternating knots and L-space knots (see Proposition~\ref{prop:L-space}).

We note here that the bound~\eqref{e:us} presents the following curious feature: it is superadditive in the knot $K$, in the sense that the bound for $K_1\#K_2$ can be strictly larger than the sum of the two bounds for $K_1$ and $K_2$. As a special case, the bound for $nK$ can give more information on $\gamma_4(K)$ than the bound for $K$. In Proposition~\ref{prop:nonsoloilvinomigliorainvecchiando} we will exhibit an example where this phenomenon actually occurs.

Using superadditivity, we can optimise the bound above as follows:
\[
\gamma_4(K) \ge \frac{\sigma(K)}2 - \ourlim(K),
\]
where $\ourlim(K)$ is defined as
\[
\lim_{n\to\infty} \frac1n\ours(n K) \le \ours(K).
\]

\subsubsection*{Organisation of the paper}
In Section~\ref{sec:dinvariants} we recall some basic facts about \spinc structures
on 3- and 4-manifold and $d$-invariants, and we state all the
results concerning them that we use in this paper.
In Section~\ref{sec:construction} we fix the notation and we
construct a cobordism $\Wo$ from a particular 3-manifold $Q$
(defined in that section) to $S^3_{-n}(K)$, which will be crucial
to deduce the bound in Equation~\eqref{e:us}.
In Section~\ref{sec:spinc} we label \spinc structures on $\Wo$,
compute their Chern classes, and understand their restrictions to $\de\Wo$.
In Section~\ref{sec:bound} we apply a twisted version of Ozsv\'ath--Szab\'o's inequality
(see Theorem~\ref{thm:OSz}) to $\Wo$ to obtain the desired bound
on $\g_4(K)$.
In Section~\ref{sec:comparison} we compare our bound to
Batson's and Ozsv\'ath--Stipsicz--Szab\'o's (see Equations~\eqref{e:batson}
and~\eqref{e:OSSz}), and we refine it using superadditivity.
Finally, in Section~\ref{sec:example}, we give an example
of a knot $K$ where the bound for $nK$ is actually better
the bound for $K$.

\subsubsection*{Acknowledgments} We would like to thank Tom Hockenhull for his encouragement and his patience; Antonio Alfieri, Fyodor Gainullin, Jen Hom, David Krcatovich, and Andr\'as Stipsicz for interesting conversations; a special thanks goes to Joshua Batson for sharing some of his unpublished computations.

\section{All you need is correction terms}
\label{sec:dinvariants}

Given an oriented manifold $M$ of dimension 3 or 4, recall that the set of \spinc
structures $\Spinc(M)$ is an affine space over $H^2(M;\Z)$.
Given an oriented 4-manifold $X$ with boundary
$\de X = Y$, the restriction map
\begin{equation}
\label{eq:restriction}
\Spinc(X) \to \Spinc(Y),
\end{equation}
is modelled over
\[
H^2(X;\Z) \to H^2(Y;\Z).
\]

To every \spinc structure $\s \in \Spinc(M)$ it is possible to associate an element
in $H^2(M;\Z)$, known as the (\emph{first}) \emph{Chern class} of $\s$, and usually
denoted by $c_1(\s)$. The map
\[
c_1 \colon \Spinc(M) \to H^2(M;\Z)
\]
is injective if and only if $H^2(M; \Z)$ has no 2-torsion.
A \spinc structure $\s \in \Spinc(M)$ is called \emph{torsion}
if $c_1(\s)$ is a torsion element in $H^2(M;\Z)$.

Let $-M$ denote the manifold $M$ endowed with the opposite orientation. There is a
canonical bijection
\[
\iota \colon \Spinc(M) \to \Spinc(-M),
\]
which is modelled over the canonical isomorphism $\iota \colon H^2(M;\Z) \to H^2(-M;\Z)$
(see~\cite[Section 1.2.3]{quadratic}).
If $\s \in \Spinc(M)$, we will denote by the same letter $\s$ the corresponding
\spinc structure on $-M$. It is worth noting that such a bijection commutes with
the restriction map (see Equation~\eqref{eq:restriction}), and that
\[
c_1(\iota(\s)) = \iota(c_1(\s)).
\]

\begin{remark}
\label{rem:label}
Let $X^4$ be the trace of the 2-handle cobordism from $S^3$ to $S^3_n(K)$,
where $K$ is a knot in $S^3$ and $n > 0$ is a positive integer. Then we
can label the \spinc structures on $X$ as follows:
we let $\s_k$ denote the unique \spinc structure on $X$ such that
\[
\scal{c_1(\s_k)}{[\Sigma]} = n + 2k,
\]
where $\Sigma$ is a Seifert surface for $K$ in $S^3\times I$, capped off with the core of the 2-handle.
From the labelling above, we derive a labelling of \spinc structures
over $S^3_n(K)$ by $\Z/n\Z$, by setting
\[
\t_{k} := \s_k |_{S^3_n(K)},
\]
where we do not make the distinction between an integer and its class modulo $n$. Here and in the following, we refer the reader to~\cite[Section 2.4]{integersurgeries} for further details.
\end{remark}

In what follows, we say that a pair $(Y,\t)$ as above, where $\t$ is a torsion \spinc structure on the 3-manifold $Y$, is a \emph{torsion} \spinc 3-manifold.

In~\cite{absolutely}, Ozsv\'ath and Szab\'o introduce a Heegaard
Floer theoretical invariant $d(Y,\t)$, called the \emph{correction term} or $d$-\emph{invariant}, associated to a pair $(Y, \t)$, where $Y$ is a
rational homology 3-sphere equipped with a \spinc structure $\t$.
In~\cite[Section 9]{absolutely}, they explain how it is possible to generalise it to invariants $d_b$ and $d_t$ (\emph{bottom} and \emph{top}) associated to a torsion \spinc 3-manifold $(Y, \t)$, where $Y$ is now a 3-manifold with \emph{standard} $\HFi$ (which is equivalent to having trivial triple cup product~\cite{Lidman}).
See also~\cite[Section 3]{LRS} for an introduction to $d$-invariants of arbitrary 3-manifolds with standard $\HFi$.
Behrens and the first author used Heegaard Floer homology with twisted coefficients to generalise this further to an invariant $\underline{d}(Y, \t)$ associated to an arbitrary torsion \spinc 3-manifold $(Y, \t)$~\cite{Behrens}.

In the case of rational homology 3-spheres we have
\[
d(Y,\t) = d_b(Y,\t) = d_t(Y,\t) = \underline{d}(Y,\t).
\]

More generally, one has the following.

\begin{theorem}[{\cite[Proposition 4.2]{absolutely}}, {\cite[Proposition 3.7]{LRS}}, and {\cite[Proposition 3.8]{Behrens}}]
\label{thm:symmetry}
Let $(Y,\t)$ be a torsion \spinc $3$-manifold, and suppose that $Y$ has standard $\HFi$.
Then, under the canonical identification $\Spinc(Y) \cong \Spinc(-Y)$,
\[
d_b(Y, \t) = - d_t(-Y, \t) = \underline{d}(Y,\t).
\]
\end{theorem}

In the rest of this section, we state the results that we need about $d$-invariants.

%
%

The following result by Ni and Wu allows us to compute
$d$-invariants for surgeries on a knot $K \subseteq S^3$ in terms of some knot invariants $V_i$,
which were first introduced in~\cite{rasmussenVi} with the name of $h_i$.
We refer to~\cite[Section 2.2]{NiWu} for the definition of $V_i$.

\begin{theorem}[{\cite[Proposition 1.6 and Remark 2.10]{NiWu}}]
\label{thm:niwu}
Given positive integers $0 \leq k < n$, then
\[
d(S^3_{n}(K), \t_{k}) = - \frac{n-(2k-n)^2}{4n} - 2 \max \set{V_k(K), V_{n-k}(K)}.
\]
\end{theorem}

Correction terms can be used to give restrictions to intersection forms of 4-manifolds bounding a given 3-manifold (compare also with~\cite[Theorem 9.15]{absolutely}).


\begin{theorem}[{\cite[Theorem 4.1]{Behrens}}]
\label{thm:OSz}
Let $(W,\s)$ be a negative semi-definite \spinc cobordism from $(Y,\t)$ to $(Y',\t')$, two torsion \spinc $3$-manifolds, such that the map $H_1(Y;\Q)\to H_1(W;\Q)$ induced by the inclusion is injective.
Then
\[
c_1(\s)^2 + b_2^-(X) \leq 4 \, \underline{d}(Y', \t') + 2\,b_1(Y') - 4\,\underline{d}(Y, \t) - 2 \, b_1(Y).
\]
\end{theorem}

\section{Notation and construction}
\label{sec:construction}

Let $K$ be a knot in $S^3$. If we consider $S^3$ as the boundary of the
4-ball $B^4$, the (orientable) slice genus $g_4$ is defined as the minimum genus
of a smooth orientable surface is $B^4$ whose boundary is $K$, and it is
a well-studied invariant of $K$. More recently, the non-orientable
slice genus $\g_4$ has been studied. We have the following definition.

\begin{definition}
\label{def:nosg} 
Given a knot $K$ in $S^3$, we define its \emph{non-orientable slice genus} as
\[
\g_4(K) = \min \set{ b_1(F) \,\middle|\,
\mbox{$F \hookrightarrow B^4$ smooth, non-orientable, $\de F =K$}
\,},
\]
where $b_1(F)$ denotes the first Betti number of $F$.
\end{definition}

\begin{remark}
With this definition of $\gamma_4$, one always has $\gamma_4(K)\ge 1$.
One could also consider the 4-dimensional crosscap number instead; this is the minimal number $h$ such that $K$ bounds a punctured $\#^h\RP^2$ in $B^4$.
The two definitions are indeed equivalent except when $K$ is slice, in which case our definition yields $\gamma_4(K)=1$, while the 4-dimensional crosscap number is 0.
We note here that, when $K$ is slice, the bound in~\eqref{e:us} is in any case $\gamma_4(K) \ge 0$, so this is in fact a bound for the crosscap number as well; this is true since, when $K$ is slice, both $\sigma(K)$ and $\ours(L)$ vanish (see Proposition~\ref{p:ours-properties}(2) below).
Our proof, however, actually uses the definition of $\gamma_4$ given above, to which therefore we stick.
\end{remark}

In~\cite{Batson}, Batson proved that the non-orientable slice genus
can be arbitrarily large. More specifically, for a non-orientable
surface $F$ as in Definition~\ref{def:nosg}, Batson gives the following
inequality (see~\cite[Theorem 4]{Batson}):
\begin{equation}
\label{eq:batson1}
b_1(F) + 2\, d(S^3_{-1}(K)) \geq \frac{e(F)}{2}.
\end{equation}
Here $d(S^3_{-1}(K))$ denotes the $d$-invariant of $S^3_{-1}(K)$ in the
unique \spinc structure, whereas $e(F)$ is the \emph{normal Euler number}
of $F$: given a non-vanishing section $s$ of the normal bundle $\ni_F$
(which always exists since $F$ deformation retracts on a 1-complex), we
let
\[
e(F) = -\lk(K, s(K)).
\]
In~\cite{Batson}, Batson combines Equation~\eqref{eq:batson1} and the
`signature' inequality
\begin{equation}
\label{eq:signature}
b_1(F) \geq \sigma(K) - \frac{e(F)}{2}
\end{equation}
to derive the bound for the non-orientable slice genus in Equation~\eqref{e:batson}.
The main result of this paper is a generalisation of Equation~\eqref{eq:batson1},
where instead of the $(-1)$-surgery along $K$ we consider $(-n)$-surgeries
for arbitrary integers $n \geq 1$. Inspired by~\cite{Batson} and~\cite{LRS},
we construct a negative semi-definite cobordism from a 3-manifold $Q$ to $S^3_{-n}(K)$,
and use Theorem~\ref{thm:OSz} to give a lower bound to $b_1(F)$.

%
\begin{figure}[t]
\labellist
\pinlabel $X$ at 210 140
\pinlabel $\widehat F$ at 163 170
\pinlabel $S^3_{-n}(K)$ at 40 170
\pinlabel $K$ at 100 83
\pinlabel $F$ at 140 30
\pinlabel $S^3$ at 0 108
\pinlabel $B^4$ at 168 8
\endlabellist
\centering
\includegraphics{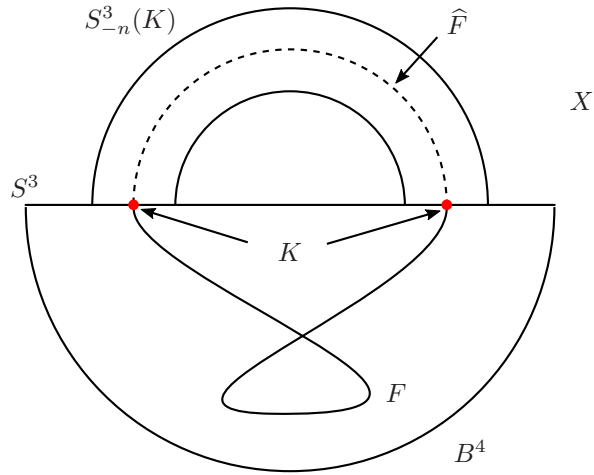}
\caption{The figure shows the 4-manifold $W$ obtained by attaching a
$(-n)$-framed 2-handle (whose trace we denote by $X$) to $B^4$ along
a knot $K \subseteq S^3$. $N = \nbd_W(\Fh)$ denotes a neighbourhood of
$\Fh$ in $W$, and $Q = \de N$.}
\label{fig:construction}
\end{figure}

We now give the details of the construction, illustrated in Figure~\ref{fig:construction}.
Let $K$ be a knot in $S^3 = \de B^4$, and let $F$ denote a smooth
non-orientable surface properly embedded in $B^4$ such that $\de F = K$.
Fix an integer $n>0$.
Let $W$ denote the 4-manifold obtained by attaching a $(-n)$-framed 2-handle to $B^4$, along $K\subset \partial B^4$.
We denote with $Y$ the boundary of $W$, i.e. $Y = S^3_{-n}(K)$.
Then the surface $F$ can be capped off with the core of the 2-handle to obtain a closed surface $\Fh \subseteq W$.
Notice that
\[
b_1(\Fh) + 1 = b_1(F) =: h.
\]
If $e = e(F)$ denotes the normal Euler number of $F$, and $e(\Fh)$ denotes the Euler number of the closed surface $\Fh$, then we have
\[
e(\Fh) = e - n.
\]
As already noticed in~\cite{Batson}, $e$ is even, because the self-intersection of $F$ in $B^4$ can be computed algebraically over $\Z/2\Z$.

Let $N = \nbd_W(\Fh)$ denote a regular neighbourhood of $\Fh$ in $W$.
We define $Q = \de N$.
Notice that $Q$ (resp.~$N$) is a circle (resp.~disc) bundle over the closed surface $\Fh \cong (\RP^2)^{\#h}$ of Euler number $e-n$.
According to the notation in~\cite[Section 2]{LRS}, we have $N \cong P_{h, e-n}$ and $Q \cong Q_{h, e-n}$, and moreover $Q$ has standard $\HFi$.

The manifold $\Wo := W \sm N$ is a cobordism between $Q$ and $S^3_{-n}(K)$.
Since the labelling of \spinc structures is better understood for positive surgeries, we consider also the manifold $-W$, obtained from $W$ by reversing the orientation;
$-W$ is the 4-manifold obtained by attaching an $n$-framed 2-handle to $B^4$ along $\ol K$.
This allows us to label the \spinc structures on $W$ and on $Y$; by a slight abuse of notation, we write $\s_k$ and $\t_k$, dropping the identifications $\Spinc(W)\cong\Spinc(-W)$ and $\Spinc(Y) = \Spinc(-Y)$.

\section{Labelling $\Spinc$ structures}
\label{sec:spinc}

\subsection{(Co-)homological computations}
\label{sec:cohomology}

The aim of this subsection is to compute $H^2(\Wo;\Z)$, in order
to understand \spinc structures on $\Wo$.
Consider the Mayer--Vietoris long exact sequence in cohomology
associated to $W = \Wo \cup_Q N$. When we do not specify it,
we assume that we are using $\Z$ coefficients.

\begin{center}
  \begin{tabular}{ c || c | c c c | c }
          & $W$  & $\Wo$ & $\sqcup$ & $N$  & $Q$ \\ \hline
    $H^0$ & $\Z$ & $\Z$  & $\oplus$ & $\Z$ & $\Z$ \\
    $H^1$ & $0$  & $0$   & $\oplus$ & $\Z^{h-1}$ & $\Z^{h-1}$ \\
    $H^2$ & $\Z$ & ?     & $\oplus$ & $\Z/2\Z$ & $\Z^{h-1} \oplus T$ \\
    $H^3$ & $0$  & $\Z$  & $\oplus$ & $0$ & $\Z$ \\
  \end{tabular}
\end{center}

The cohomology of $W$ can be easily obtained by recalling that $W$ is
constructed by attaching a 2-handle on a $B^4$. The cohomology of $N$ is also straightforward,
since $N$ deformation retracts on $\Fh = (\RP^2)^{\# h}$.
As for $Q$, its cohomology can be deduced from~\cite[Lemma 2.1]{LRS},
and it is written in the table above. $T$ is the torsion subgroup of
$H_1(Q)$, which is, according to~\cite[Lemma 2.1]{LRS},
\[
T = \begin{cases}
\Z/2\Z \oplus \Z/2\Z & \mbox{if $e(\Fh)$ is even,}\\
\Z/4\Z & \mbox{if $e(\Fh)$ is odd.}\\
\end{cases}
\]
In both cases, the map $H^2(N) \cong \Z/2\Z \to T$ is non-trivial.
From the cohomology groups that we already know (and the fact that
the map $H^1(N) \to H^1(Q)$ is an isomorphism)
we can deduce almost all the cohomology groups of $\Wo$. $H^2(\Wo)$ will depend on the parity of $e(\Fh)$, according to the following lemma.

\begin{lemma}
\label{lem:H2Wo}
We have that
\[
H^2(\Wo) = \begin{cases}
\Z^{h} \oplus \Z/2\Z & \mbox{if $e(\Fh)$ is even,}\\
\Z^{h} & \mbox{if $e(\Fh)$ is odd.}\\
\end{cases}
\]
\end{lemma}
\begin{proof}
From the long exact sequence above we have an exact sequence
\[
0 \to \Z \to H^2(\Wo) \to \Z^{h-1} \oplus \Z/2\Z \to 0,
\]
regardless of the parity of $e(\Fh)$. The two possible extensions
are $\Z^{h} \oplus \Z/2\Z$ and $\Z^{h}$. If we consider the reduction
modulo 2, which maps $H^2(\Wo)$ to $H^2(\Wo;\F_2)$, the two possible
extensions are mapped respectively to $\F_2^{h+1}$ and $\F_2^h$.
Therefore, in order to understand $H^2(\Wo)$, it is enough to determine
the rank of $H^2(\Wo; \F_2)$.

Consider the Mayer--Vietoris long exact sequence in homology associated
to $W = \Wo \cup_Q N$, with $\F_2$ coefficients. Since the coefficient
ring is a field, homology and cohomology are dual to each other.
As in the previous case, the homologies of $W$ and $N$ are quite straightforward to compute.
According to~\cite[Proof of Lemma 2.1]{LRS}, $H_1(Q;\F_2) \cong H_2(Q;\F_2) \cong \F_2^{h+1}$
if $e(\Fh)$ is even, and $H_1(Q;\F_2) \cong H_2(Q;\F_2) \cong \F_2^{h}$ if $e(\Fh)$ is odd.

\begin{center}
  \begin{tabular}{ c || c | c c c | c }
          & $Q$           & $\Wo$  & $\sqcup$ & $N$      & $W$ \\ \hline
    $H_3$ & $\F_2$        & $\F_2$ & $\oplus$ & $0$      & $0$ \\
    $H_2$ & $H_2(Q;\F_2)$ & ?      & $\oplus$ & $\F_2$   & $\F_2$ \\
    $H_1$ & $H_1(Q;\F_2)$ & ?      & $\oplus$ & $\F_2^h$ & $0$ \\
    $H_0$ & $\F_2$        & $\F_2$ & $\oplus$ & $\F_2$   & $\F_2$ \\
  \end{tabular}
\end{center}
Consider the connecting morphism
\[
\de \colon H_2(W;\F_2) \to H_1(Q;\F_2).
\]
If $\alpha = [\Fh]$ is the generator of $H_2(W;\F_2) \cong \F_2$ and $\gamma$ is the class of a circle fibre in $H_1(Q;\F_2)$, then
\[
\de\alpha = e(\Fh) \cdot \gamma \in H_1(Q; \F_2).
\]
It follows that the map $\de$ is trivial when $e(\Fh)$ is even.

The map $\de \colon H_2(W;\F_2) \to H_1(Q;\F_2)$ is trivial also
when $e(\Fh)$ is odd. This is because ---according to~\cite[Lemma 2.1]{LRS}---
when $e(\Fh)$ is odd the homology class of the fibre
is divisible by 2 in $H_1(Q)$,
so its reduction modulo $2$ is $0$.

From this we deduce that $H_2(\Wo;\F_2) \cong \F_2^h$ if $e(\Fh)$ is odd, and $H_2(\Wo;\F_2) \cong \F_2^{h+1}$ if $e(\Fh)$ is even, and hence we conclude the proof of the lemma.
\end{proof}

\subsection{Intersection form}
In this section we study the intersection forms on $H_2(W)$ and $H^2(W)$.
\begin{lemma}
The intersection form $Q_W$ on $H_2(W) \cong \Z$ is given by $Q_W = (-n)$.
The intersection form $Q^W$ on $H^2(W) \cong \Z$ is given by $Q^W = (-\frac{1}{n})$.
\end{lemma}
\begin{proof}
The intersection form on $H_2(W)$ is $(-n)$ because the 4-manifold $W$
is obtained by attaching a $(-n)$-framed 2-handle to $B^4$.

The intersection form on $H^2(W)$ can be worked out by considering the
following portion of the long exact sequence in homology associated to
the couple $(W, Y)$, where $Y= S^3_{-n}(K)$:
\[
0 \to H_2(W) \to H^2(W) \to H_1(Y) \to 0.
\]
Such a short exact sequence is isomorphic to
\[
0 \to \Z \to \Z \to \Z/n\Z \to 0.
\]
The generator of $H_2(W)$ is mapped to $n$ times the generator
of $H^2(W)$, so the intersection form on $H^2(W)$ is represented
by the matrix $(-\frac{1}{n})$.
\end{proof}

It is also worth noting that for each $c \in H^2(W)$, $c|_{\Wo}$ restricts to a torsion \spinc structure on both boundary components, and therefore it makes sense to consider its square.
We claim that:
\begin{equation}
\label{eq:intersection}
Q^{\Wo}(c|_{\Wo}) = Q^W(c).
\end{equation}
Indeed, the class $nc \in H^2(W)\cong H_2(W,Y)$ is in the image of the map $H_2(W)\to H_2(W,Y)$.
Now, the map $\iota \colon H_2(\Wo) \to H_2(W)$ is surjective: this comes from the  Mayer--Vietoris sequence for $W = \Wo\cup N$, since the connecting morphism $\partial: H_2(W)\to H_1(Q)$ vanishes (see the proof of Lemma~\ref{lem:H2Wo}) and $H_2(N) = 0$.

Therefore, there is an element $d \in H_2(\Wo)$ such that $\iota(d) = nc$.
The elements $d$ and $nc$ can be represented by some copies of a surface $S \subseteq \Wo$.
The squares $Q^{\Wo}(d)$ and $Q^W(nc)$ can be computed as the algebraic self-intersection $S\cdot S$ of $S$,
which in turn can be computed in an arbitrarily small neighbourhood of $S$.

\subsection{Spin$^c$ structures}

Recall (see Remark~\ref{rem:label}) that \spinc structures on $-W$
are labelled by integers as follows:
\[
\scal{c_1(\s_k)}{[\Sigma]} = 2k + n.
\]
By symmetry we also get a labelling for $\Spinc(W)$, and we still
denote the \spinc structures on $W$ by $\s_k$.
$\s_k$ and $\s_{k'}$ restrict to the same \spinc structure on $Y$
if and only if $n \,|\, (k-k')$. In such a case we denote the restriction
to $Y$ by $\t_{k} = \t_{k'}$.

It is worth noting that we have isomorphisms $H^2(W) \cong \Z$ and
$H^2(Y) \cong \Z/n\Z$ such that, under these identifications, the
restriction map is the usual projection $\Z \to \Z/n\Z$, and $c_1(\t_{k}) \equiv 2k\pmod n$.

In order to apply Theorem~\ref{thm:OSz}, we need a \spinc structure
on the cobordism $\Wo$ that restricts to a torsion \spinc structure
on $Q$. Therefore, we introduce the following notation:

\begin{definition}
Given a 4-manifold $X$, we define $\Spinctor(X)$ to be the subset of
$\Spinc(X)$ of elements that restrict to torsion \spinc structures
on $\de X$.
\end{definition}

Notice that in our case $\Spinctor(\Wo)$ is given by all \spinc
structures that restrict to torsion \spinc structures on $Q$,
because all \spinc structures on $Y$ are already torsion.
We will now give a classification of $\Spinctor(\Wo)$ in the case
of $e(\Fh)$ odd (or, equivalently, $n$ odd).

%
%
%

%
%
%
%
%

\subsection{The case $e(\Fh)$ odd}

By Lemma~\ref{lem:H2Wo} we have that $H^2(\Wo) \cong \Z^h$.
From the Mayer--Vietoris exact sequence associated to $W = \Wo \cup_Q N$ we find:
\[
\xymatrix{
0 \ar[r] & H^2(W) \ar[r] \ar[d]^\vsimeq & H^2(\Wo) \oplus H^2(N) \ar[r] \ar[d]^\vsimeq & H^2(Q) \ar[r] \ar[d]^\vsimeq & 0 \\
0 \ar[r] & \Z \ar[r]^-{\a} & \Z^h \oplus \Z/2\Z \ar[r]^-{\b} & \Z^{h-1} \oplus \Z/4\Z \ar[r] & 0 
}
\]
We have that $\a(1) = (c, 1)$ for some nonzero $c \in \Z^h$,
otherwise the quotient would contain a $\Z/2\Z$ summand.
Then we have that
\[
\Z^{h-1} \oplus \Z/4\Z \cong \frac{\Z^h \oplus \Z/2\Z}{\spn{(c,1)}} \cong \frac{\Z^h }{\spn{2c}}.
\]
This implies that $c = 2d$, where $d \in \Z^h$ is a primitive element.
We denote by $x \in H^2(\Wo)$ the element that corresponds to $d$, and
we let $\A = \spn{x} \subseteq H^2(\Wo)$. Therefore, $\Spinctor(\Wo)$
is an affine space over $\A$.

It follows from the exact sequence above that the image of the map
\[
\Spinc(W) \to \Spinc(\Wo)
\]
is contained inside $\Spinctor(\Wo)$. Moreover, the map is modelled
on the map
\[
H^2(W) \cong \Z \xrightarrow{\cdot2} \Z \cong \A.
\]
It follows from the naturality of the first Chern class that $c_1(\s_k|_{{\Wo}}) = (2n + 4k)x$:
\[
\xymatrix{
\Spinc(W) \ar[r] \ar[d]^{c_1} & \Spinctor(\Wo) \ar[d]^{c_1} \\
n + 2\Z \ar[r]^-{\cdot2} & 2n + 4\Z \subseteq \A\\
}
\]
The Chern classes of all \spinc structures in $\Spinctor(\Wo)$
form the subset $2n + 2\Z = 2\Z \subseteq \Z \cong \A$.
This motivates the following definition.

\begin{definition}
We define $\rs_k \in \Spinctor(\Wo)$ to be the \spinc structure
on $\Wo$ that restricts to a torsion \spinc structure on $Q$ and
that satisfies
\[
c_1(\rs_k) = (2n + 2k) x.
\]
\end{definition}

\begin{remark}
\label{rem:restrictiontoQ}
It follows from the computations above that
\begin{align*}
\Spinc(W) &\to \Spinctor(\Wo)\\
\s_k &\mapsto \rs_{2k}
\end{align*}
and that $\rs_k \in \Spinctor(\Wo)$ extends to a \spinc structure on $W$ if and only
if $k$ is even.
\end{remark}

We now want to understand the restriction of the \spinc structure $\rs_k$ to $Y$.
This is done in the following lemma. Instead of $W$, we use $W_n=-W$ and $S^3_n(\mK)=-Y$ to label
the \spinc structure, so we can stick to the usual positive surgery conventions.

\begin{lemma}
\label{lem:restrictiontoY}
For all $k \in \Z$ we have that
\[
\rs_{2k}|_{S^3_n(\mK)} = \t_k \,\,\, \mbox{and} \,\,\, \rs_{n+2k}|_{S^3_n(\mK)}=\t_k.
\]
\end{lemma}

\begin{proof}
Consider the following commutative diagram:
\[
\xymatrix{
H^2(W_n) \ar[rr] \ar[dr]_-{\pi} & & \A \ar[ld]^-{r} & & \Z \ar[rr]^{\cdot2} \ar[dr]_-{\pi} & & \Z \ar[ld]^-{r} \\
 & H^2(S^3_n(\mK)) & & & & \Z/n\Z & \\
}
\]
Recall that we chose isomorphisms $H^2(W_n)\cong \Z$ and $H^2(S^3_n(\mK)) \cong \Z/n\Z$ such that $\pi(1) = 1 \in \Z/n\Z$.
Then
\[
c_1(\t_k) = \pi(c_1(\s_k)) = n+2k = 2k.
\]
Since $n$ is odd, $2$ is invertible modulo $n$, so every \spinc
structure on $S^3_n(\mK)$ is determined by its first Chern class.

By the naturality of the Chern class we have that for every $k \in \Z$,
the following diagram commutes:
\[
\xymatrix{
c_1(\s_k) \ar@{|->}[rr] \ar@{|->}[dr]_-{\pi} & & c_1(\rs_{2k}) \ar@{|->}[ld]^-{r} & & n+2k \ar@{|->}[rr]^{\cdot2} \ar@{|->}[dr]_-{\pi} & & 2n+4k \ar@{|-->}[ld]^-{r} \\
 & c_1(\t_k) & & & & 2k & \\
}
\]
From this, we obtain that $\rs_{2k}|_{S^3_n(\mK)} = \t_k$.

For the case of $\rs_{n+2k}$, recall that $c_1(\rs_{n+2k}) = 4n+4k$.
From the commutativity of the diagram below we deduce that $\rs_{n+2k}|_{S^3_n(\mK)} = \t_k$.
\[
\begin{gathered}[b]
\xymatrix{
2n+2k \ar@{|->}[rr] \ar@{|->}[dr]_-{\pi} & & 4n+4k \ar@{|-->}[ld]^-{r}\\
 & 2k & \\
}\\[-\dp\strutbox]
\end{gathered}
\qedhere
\]
\end{proof}

\subsection{The case $e(\Fh)$ even}
\label{sec:even}

When $e(\Fh)$ is even, $H^2(\Wo) \cong \Z^h \oplus \Z/2\Z$ by Lemma~\ref{lem:H2Wo}.
One can check that $\Spinctor(\Wo)$ is an affine space over a submodule
\[
\Z \oplus \Z/2\Z \subseteq \Z^h \oplus \Z/2\Z,
\]
where the $\Z$ summand is generated by a primitive element $x$.
One can then define
\[
\rs_k := \s_k|_{\Wo} \in \Spinctor(\Wo),
\]
and, if $\g$ denotes the generator of the $\Z/2\Z$ summand,
\[
\td{\rs_k} := \rs_k +\g \in \Spinctor(\Wo).
\]
One can check that $\td{\rs_k}$ restricts to $Q$ to a non-extendible
\spinc structure $\td\t$, and to $Y$ to the \spinc structure $\t_{k +\frac{n}{2}}$.
Moreover, we have that
\[
c_1(\td{\rs_k})^2 = - \frac{(n+2k)^2}{n}.
\]
Note that $n$ is even because so is $e(\Fh)$, so $k+\frac{n}{2}$ is an integer.

\section{A bound for the non-orientable slice genus}
\label{sec:bound}

We now prove Theorem~\ref{t:main}, that we restate here.
Recall that we have defined $\ours(K)$ to be the quantity $\min_{m\ge 0}\{m+2V_m(\overline K)\}$.

{\renewcommand{\thethm}{\ref{t:main}}
\begin{thm}
For every knot $K$ in $S^3$,
{
\renewcommand{\theequation}{\ref{e:us}}
\begin{equation}
\gamma_4(K) \ge \frac{\sigma(K)}2 - \ours(K).
\end{equation}
\addtocounter{equation}{-1}
}
\end{thm}
\addtocounter{thm}{-1}
}
\begin{proof}
Choose an odd integer $n>0$, and let $k$ be any integer.
We denote by $[k]$ the representative for the residue class of $k$ modulo $n$ such that $0 \leq [k] < n$.
By Remark~\ref{rem:restrictiontoQ} and Lemma~\ref{lem:restrictiontoY}, the \spinc structure $\rs_{n+2k}$ restricts to a non-extendible \spinc structure on $Q$, that we denote by $\td\t$, and to $\t_{k}$ on $Y$.

We apply Theorem~\ref{thm:OSz} to the cobordism $(\Wo,\rs_{n+2k})$ turned upside down, i.e. seen as a cobordism from $(-Y,\t_k)$ to $(-Q,\td\t)$: the assumption that the map $H_1(Y;\Q)\to H_1(\Wo;\Q)$ be injective is automatically satisfied, since $Y$ is a rational homology sphere.
The inequality of Theorem~\ref{thm:OSz} then reads as follows:
\begin{equation}
\label{eq:OSz1}
c_1(\rs_{n+2k})^2 + b_2^-(\Wo) \leq 4\,\underline{d}(-Q,\td\t) + 2b_1(Q) - 4\,d(-S^3_{-n}(K), \t_{k}).
\end{equation}

We now compute each term of Equation~\eqref{eq:OSz1}. We have that $b_2^-(\Wo) = 1$ and $b_1(Q) = h-1$. Moreover,
\[
c_1(\rs_{n+2k})^2 = \left((4n+4k)x\right)^2 = -\frac{1}{4n}\cdot(4n+4k)^2 = -\frac{4}{n}\cdot(n+k)^2,
\]
where we used the fact that $Q^{\Wo}(2x,2x) = -\frac{1}{n}$.

As for the $d$-invariant of $S^3_{-n}(K)$, by Theorems~\ref{thm:symmetry} and~\ref{thm:niwu} we have
\begin{equation*}
d(-S^3_{-n}(K), \t_{k}) = d(S^3_n(\mK), \t_{k}) = -\frac{n - (2[k] - n)^2}{4n} - 2\,\max\set{\ol{V}_{[k]}, \ol{V}_{n-[k]}},
\end{equation*}
where we set $\ol{V}_i :=V_i(\mK)$.

Finally, by~\cite[Theorem 5.1]{LRS} and~Theorem~\ref{thm:symmetry} above, we have that
\[
\underline{d}(-Q,\td\t) = -d_t(Q, \td\t) = -\left(\frac{e(\Fh)-2}{4} + a\right) \leq -\frac{e-n-2}{4}.
\]

Therefore, Equation~\eqref{eq:OSz1} becomes
\begin{equation*}
-\frac{4}{n}\cdot(n+k)^2 + 1 \leq
\frac{n - (2[k] - n)^2}{n} + 8\,\max\set{\ol{V}_{[k]}, \ol{V}_{n-[k]}} - (e-n-2) + 2h-2,
\end{equation*}
which can be re-written as follows:
\begin{equation}
\label{eq:OSz2}
2h + 8\,\max\set{\ol{V}_{[k]}, \ol{V}_{n-[k]}} \geq
e - n - \frac{4(n+k)^2 - (2[k]-n)^2}{n}
\end{equation}
By combining it with Equation~\eqref{eq:signature} as in~\cite{Batson},
we obtain:
\begin{equation}
\label{eq:combo}
4h + 8\,\max\set{\ol{V}_{[k]}, \ol{V}_{n-[k]}} \geq
2\sigma(K) - n - \frac{4(n+k)^2 - (2[k]-n)^2}{n}.
\end{equation}

Given a fixed integer $m \geq 0$, it is not difficult to check
that the best bound for $h$ coming from Equation~\eqref{eq:combo}
and involving $\ol{V}_m$ is obtained by setting $n=2m+2j+1$
and $k = -n \pm m$ (where $j$ is an arbitrary non-negative integer).
The bound for $\g_4(K)$ that we obtain in this case is then:
\begin{equation}
\label{eq:GM}
\g_4(K) \geq \frac{\sigma}{2} - m - 2 \ol{V}_m.
\end{equation}
By taking the maximum over $m \geq 0$ we conclude the proof of the theorem.
\end{proof}

\begin{remark}
By setting $m=0$ in Equation~\eqref{eq:GM}, we obtain exactly Batson's
inequality~\eqref{e:batson}.
\end{remark}

\begin{remark}
\label{rem:sharpness}
For every $m \geq 0$, the bound in Equation~\eqref{eq:GM} is sharp, in the sense that for each $m$ there exists a knot $K_m$ such that $\gamma_4(K_m) = \frac{\sigma(K)}2-m-V_m(\overline K)$.
The knot $K_0 = T_{3,-4}$ exhibits that the inequality is sharp for $m=0$, as already shown by Batson~\cite{Batson}.

For $m\ge 1$, consider the torus knot $K=T_{3,-5}$, whose signature is $8$.
Since $\mK = T_{3,5}$ is a \emph{positive} torus knot, hence an
L-space knot, the invariants $V_i(T_{3,5})$ coincide with
the torsion coefficients~\cite[Corollary 7.5]{absolutely}:
\[
V_i(\mK) = \sum_{j>0} j \, a_{j+i},
\]
where
\[
\Delta_{\mK}(t) = a_0 + \sum_{j>0} a_j \left( t^j + t^{-j} \right)
\]
is the Alexander polynomial of $\mK$. One can explicitly compute that,
for $\mK = T_{3,5}$,
\[
\Delta_{T_{3,5}}(t) = t^4 - t^3 + t - 1 + t^{-1} - t^{-3} + t^{-4}.
\]
It follows that $V_1(\mK) = 1$ and that Equation~\eqref{eq:GM} for $m=1$
gives
\[
\g_4(K) \geq \frac{8}{2} - (1 + 2) = 1.
\]
Since $\mK$ bounds a Moebius band in $B^4$, as shown in Figure~\ref{fig:T35} (see also~\cite[Section 4]{Batson}),
it follows that~\eqref{eq:GM} is sharp for $m=1$.

\begin{figure}[t]
\labellist
\pinlabel {\Huge $\rightsquigarrow$} at 110 23
\pinlabel {\LARGE $=$} at 243 23
\endlabellist
\centering
\includegraphics{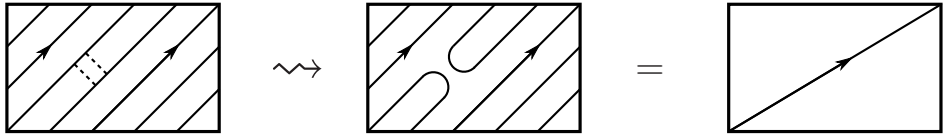}
\caption{The figure shows a non-orientable cobordism of genus 1 from $T_{3,5}$
to the unknot. The rectangle above represents a torus, which is embedded in $S^3$
in the standard way. The (unoriented) band surgery above carries $T_{3,5}$ to
the unknot.}
\label{fig:T35}
\end{figure}

For all $m > 1$, consider the knot $mK$, the connected sum of $m$ copies of $K$.
Recall from~\cite[Proposition 6.1]{BCG2} that the sequence
$\{V_i(K)\}$ satisfies the following subadditivity property:
$V_{k+l}(K\# L) \le V_k(K)\# V_l(L)$ for each pair $(k,l)$ of non-negative
integers and each pair $(K,L)$ of knots. By subadditivity of $\g_4$,
subadditivity of the $V_i$, and additivity of the signature, we obtain
\begin{align*}
m = m\g_4(K) & \geq \g_4(mK) \\
& \geq \frac{\sigma(mK)}{2} - (m + 2V_m(m\mK))\\
& \geq m \left( \frac{\sigma(K)}{2} - (1 + 2V_1(\mK)) \right) = m.
\end{align*}
It follows that all the inequalities above are actually \emph{equalities},
and that therefore~\eqref{eq:GM} is sharp for every $m \geq 1$.
\end{remark}

\begin{remark}
In the proof of Theorem~\ref{t:main} we only considered surgery with
some odd framing $n > 0$. If we considered the case of even $n$, and
applied Theorem~\ref{thm:OSz} to the torsion \spinc structure $\td{\rs_k}$
(defined in Section~\ref{sec:even}), we would have obtained exactly the
same bound as Equation~\eqref{eq:GM} for all $m \geq 0$.
\end{remark}

\section{Comparison to other bounds}
\label{sec:comparison}

In this section we study some properties of the functions $\ours$ and $\ourlim$ defined in the introduction, and discuss the relationship between the bounds given by~\eqref{e:batson},~\eqref{e:OSSz}, and~\eqref{e:us}.

\begin{proposition}\label{p:ours-properties}
The invariant $\ours$ is a concordance invariant, with values in the non-negative integers. It has the following properties:
\begin{enumerate}
\item $0 \le \ours(K) \le \min\{\nu^+(\ol K), 2V_0(\ol K)\}$;
\item $\ours(K) = 0$ if and only if $V_0(\ol K) = \nu^+(\ol K) = 0$; in particular, if $K$ is slice, $\ours(K) = 0$;
\item if there is an orientable genus-$g$ cobordism from $K_1$ to $K_2$, then $|\ours(K_1)-\ours(K_2)|\le g$;
\item if $K_+$ is obtained from $K_-$ by performing a crossing change from negative to positive, then $\ours(K_-)-1\le \ours(K_+)\le \ours(K_-)$;
\item for every two knots $K_1, K_2$, $\ours(K_1\#K_2) \le \ours(K_1) + \ours(K_2)$.
\end{enumerate}
\end{proposition}

We remark here that, in particular, $\ours$, much like $\nu^+$, and by constrast with $\sigma$ and $\upsilon$, does \emph{not} induce a homomorphism from the concordance group to the integers.

\begin{proof}
The sequence $\{V_i(K)\}_i$ is a concordance invariant, hence so is $\ours$; 
moreover, the quantity $m+2V_m(\ol K)$ is a non-negative integer for each $m$, and hence so is $\ours(K)$.

\begin{enumerate}
\item When $m=0$, $m+2V_m(\ol K) = 2V_0(\ol K)$, while for $m=\nu^+(\ol K)$, $m+2V_m(\ol K) = \nu^+(K)$.
By definition, $\ours(K) \le 2V_0(\ol K)$ and $\ours(K) \le \nu^+(\ol K)$.

\item Observe that $m+2V_m(\ol K)$ is always strictly positive if $m>0$; hence, if $\ours(K) = \min_m \{m+2V_m(\ol K)\} = 0$, the minimum can only be attained at $m=0$, and in that case $V_0(\ol K) = 0$, which implies $\nu^+(\ol K) = 0$. The converse is obvious.\\
When $K$ is slice, $\nu^+(\ol K) = 0$, and hence $\ours(K)$ vanishes, too.

\item{By {\cite[Lemma 5.1]{BCG2}} we have that, under the given assumptions, $V_{m+g}(K_1) \le V_m(K_2)$ for each non-negative integer $m$. It follows that $m+g+2V_{m+g}(K_1) \le m+2V_m(K_2)+g$, hence, minimising over $m$,
\[
\ours(K_1) \le \min_{m'\ge g}\{m'+2V_{m'}(K_2)\} \le \ours(K_2)+g.
\]
Exchanging the roles of $K_1$ and $K_2$, we obtain the symmetric inequality.}

\item Observe that there is a genus-$1$ cobordism from $K_-$ to $K_+$, obtained by smoothing the double point in the trace of the crossing change homotopy. Thus, point (3) above shows that $\ours(K_-)-1\le \ours(K_+)$. Using~\cite[Theorem 6.1]{BorodzikHedden} we also obtain:
\[V_m(\overline{K_+}) \le V_{m}(\overline{K_-}),\]
from which, for each $m\ge 0$,
\[m+2V_m(\overline{K_+}) \le m+2V_m(\overline{K_-}),\]
and minimising over all values of $m$ yields the desired inequality.

\item For each $k,l$ non-negative integers, $V_{k+l}(K_1\#K_2) \le V_k(K_1) + V_l(K_2)$ by~\cite[Proposition 6.1]{BCG2}, hence
\begin{align*}
\ours({K_1}\# {K_2}) &= \min_n\{n+2V_{n}(\ol{{K_1}\#{K_2}})\}\\
&\le \min_n \min_{k+l=n}\{k+l+2V_k(\ol {K_1})+2V_l(\ol {K_2})\} \\
&= \min_k\{k+2V_k(\ol {K_1})\} + \min_l\{l+2V_l(\ol {K_2})\} \\
&= \ours({K_1}) + \ours({K_2}).\qedhere
\end{align*}
\end{enumerate}
\end{proof}

We will compare our bound with~\eqref{e:OSSz} obtained by Ozsv\'ath--Stipsicz--Szab\'o, and in order to do so we need to compare $\upsilon(K)$ with $\ours(K)$.
We say that a knot is \emph{Floer-thin} if its knot Floer homology is supported on the diagonal $i-j = -\tau(K)$.

\begin{proposition}
\label{prop:L-space}
When $K$ is a Floer-thin knot with $\tau(K)\ge 0$ or an L-space knot, then $\ours(\overline{K}) = -\upsilon(K)$ and $\ours(K) = 0$.
\end{proposition}

In particular, the bound given by~\eqref{eq:GM} for both $K$ and $\overline K$ is at most as strong as the one given by $\upsilon$, when $K$ is an L-space knot or an alternating knot.

\begin{proof}
Recall that for a Floer-thin knot $K$ with $\tau(K) = \pm n$, we have $V_i(K) = V_i(T_{2,\pm (2n+1)})$~\cite[Equation (8)]{AG}, and hence $\ours(K) = \ours(T_{2,\pm(2n+1)})$. Analogously, it follows from~\cite[Theorem 1.14]{upsilon} that $\upsilon(K) = \upsilon(T_{2,\pm(2n+1)})$.
It follows that it is enough to prove the statement for L-space knots.

When $K$ is an L-space knot, then a direct computation from the knot Floer complex shows that $V_i(\ol K) = 0$ for every $i$; hence $\ours(K) = 0$.
On the other hand, Borodzik and Hedden have shown in~\cite[Proposition 4.6]{BorodzikHedden} that
\[
\upsilon(K) = \Upsilon_K(1) = -\min_{n} \set{n + 2V_n(K)} = -\ours(\mK),
\]
as desired.
\end{proof}

In the case of Floer-thin knots we can actually say more about $\ours$.

\begin{proposition}
\label{prop:Floer-thin}
If $K$ is a Floer-thin knot with $\tau(K)\ge0$, then we have
\[
\ours(\mK) = \ni^+(K) = \tau(K) = -\upsilon(K).
\]
If, additionally, $K$ is quasi-alternating, then $\ours(\mK) = -\sigma(K)/2$, and in this case the bounds~\eqref{e:OSSz} and
\eqref{e:us} -- applied to $K$ and $\mK$ -- yield
\[
\gamma_4(K) \geq 0.
\]
\end{proposition}

\begin{proof}
By~\cite[Equation (8)]{AG}, we know that the minimum of $\set{m + 2V_m(K)}$
is attained at $m = \tau(K) = \ni^+(K)$. This implies at once that $\ours(\mK) = \tau(K)$.
The equality with $-\upsilon(K)$ follows from Proposition~\ref{prop:L-space}.

When $K$ is quasi-alternating, $\tau(K) = -\sigma(K)/2$, and the second part of the statement readily follows.
\end{proof}

In many instances, the bound given by $\upsilon$ is better than the one given by $\ours$; this is true, for example, for many knots of the form $K_1\# \ol{K_2}$, where $K_1$ and $K_2$ are L-space knots.

\begin{example}
\label{ex:T25T56}
Consider the two knots $K_1 = T_{2,3}$, $K_2 = T_{5,6}$, and let $K=K_1\# \ol K_2$.
One computes $\sigma(K_1) = -2$, $\sigma(K_2) = -16$, $\upsilon(K_1) = -1$ and $\upsilon(K_2) = -6$.
Using the techniques from~\cite{Krcatovich} as in~\cite{BCG2}, we can also compute
$\ours(K) = 6$ and $\ours(\mK) = 0$.

It follows that the bound given by~\eqref{e:us}, applied to both $K$ and $\ol K$, gives $\gamma_4(K) \ge 1$,
while the bound given by~\eqref{e:OSSz} is $\gamma_4(K) \ge 2$.
\end{example}

As a consequence of Proposition~\ref{p:ours-properties}, we deduce the following interesting feature of $\ours$.

\begin{corollary}
The invariant $\ours(K)$ is subadditive. In particular, the following identity holds:
\[
\lim_{n\to \infty} \frac1n\ours(nK) = \inf_n\frac1n\ours(nK).
\]
\end{corollary}

\begin{proof}
By property (5) of Proposition~\ref{p:ours-properties}, the function $n\mapsto \ours(nK)$ is subadditive,
in the sense that $\ours(aK+bK) \le \ours(aK) + \ours(bK)$ for every $a,b\ge 0$.
The existence of the limit follows from Fekete's lemma~\cite{Fekete}.
\end{proof}


\begin{definition}
\label{def:ourlim}
We call $\ourlim(K) = \lim_n\frac1n\ours(nK)$.
\end{definition}

We now introduce the \emph{stable} non-orientable 4-genus $\gamma_4^{\rm st}(K)$ of $K$, i.e. the limit $\lim_{n\to\infty} \frac1n\gamma_4(nK)$.
Notice that the limit exists since the sequence $(\gamma_4(nK))_n$ is subadditive, and that $\gamma_4^{\rm st}(K) \le \gamma_4(K)$.

\begin{theorem}\label{t:ourlimprop}
The invariant $\ourlim(K)$ is a concordance invariant of $K$, and it descends to a subadditive, homogeneous function $\ourlim \colon \mathcal C \to \mathbb R_{\ge 0}$. Additionally:
\begin{enumerate}
\item $\gamma^{\rm st}_4(K) \ge \frac{\sigma(K)}2 - \ourlim(K)$;
\item if there is an orientable genus-$g$ cobordism between $K_1$ and $K_2$, then $|\ourlim(K_1)-\ourlim(K_2)| \le g$;
\item if there is a crossing change (from negative to positive) from $K_-$ to $K_+$, then $\ourlim(K_-)-1 \le \ourlim(K_+) \le \ourlim(K_-)$.
\end{enumerate}
\end{theorem}

As an immediate corollary to the theorem, we get the following:
\begin{corollary}
If the inequality in Theorem~\ref{t:main} is sharp, then $\gamma_4(nK) = n\gamma_4(K)$ for each $n$; in particular $\gamma^{\rm st}_4(K) = \gamma_4(K)$.
\end{corollary}

As remarked for $\ours$ above, $\ourlim$ is not a homomorphism, since it takes only non-negative values.
Note also that $\ourlim$ is not identically 0, since, by Proposition~\ref{prop:Floer-thin} applied to $nK$ for all $n\geq0$, $\ourlim(K)$ coincides with $\sigma(K)/2$ for Floer-thin knots with positive signature.

Also, by definition, $\ourlim(K) \le \ours(K)$, and in particular the bound for $\gamma_4^{\rm st}(K)$ given by $\ourlim$ can be \emph{better} than the bound given by $\ours$ on $\gamma_4(K)$ (see Proposition~\ref{prop:nonsoloilvinomigliorainvecchiando} for an example).
This is by contrast with the bound given, for example, by $\tau$, $s$, or $\nu^+$ on the stable orientable slice genus: the first two are linear, while the third is \emph{sub}linear in $K$~\cite[Theorem 1.4]{BCG2}.

\begin{proof}[Proof of Theorem~\ref{t:ourlimprop}]
The invariant $\ourlim$ is a concordance invariant, since $\ours$ is, and it takes non-negative values, since $\ours$ does.
Moreover, it is subadditive by construction:
\begin{align*}
\ourlim(K\#L) &= \lim_n\left\{\frac1n\ours(n(K\#L))\right\} \le \lim_n \left\{\frac1n(\ours(nK)+\ours(nL))\right\} = \\
&= \lim_n\left\{ \frac1n\ours(nK)\right\}+\lim_n\left\{\frac1n\ours(nL)\right\} = \ourlim(K) + \ourlim(L),
\end{align*}
where the inequality follows from the subadditivity of $\ours$ (Property (5) of Proposition~\ref{p:ours-properties}).

It is also homogeneous, in the sense that $\ourlim(nK) = n\ourlim(K)$: 
\[
\ourlim(nK) = \lim_{m} \frac{1}{m} \ours(mnK) = n \lim_{m} \frac{1}{mn} \ours(mnK) = n \lim_{m'} \frac{1}{m'} \ours(m'K) = n \ourlim(K).
\]

\begin{enumerate}
\item Applying~\eqref{eq:GM} for $nK$ we obtain, for each $n\ge1$:
\[
\gamma_4(nK) \ge \frac{\sigma(nK)}2 - \ours(nK) = n\frac{\sigma(K)}{2} - \ours(nK),
\]
from which
\[
\gamma^{\rm st}_4(K) = \lim_n \frac{\gamma_4(nK)}n \ge \frac{\sigma(K)}2 - \lim_n \frac{\ours(nK)}n = \frac{\sigma(K)}2 - \ourlim(K).
\]
\end{enumerate}
Properties (2) and (3) follow immediately from the corresponding properties of $\ours$, stated in Proposition~\ref{p:ours-properties} above.
\end{proof}

\section{An example}
\label{sec:example}
An interesting feature of $\ourlim$ is that --- by contrast with $\ours$ --- it can attain
non-integer values, as we shall see presently.

To this end, we study an example in detail: we show that $\ourlim(T_{2,3}-T_{5,6}) = \frac{26}5$.
Before doing so, we recall some facts about Krcatovich's reduced knot Floer complex.

In~\cite{Krcatovich}, Krcatovich associates to each knot $J \subset S^3$ a reduced
version of the knot Floer complex, denoted by $\rCFKm(J)$.
The reduced knot Floer complex for L-space knots is of a particularly simple form, in that it only consists of a single \emph{tower},
i.e. it is isomorphic to $\mb F[U]$ as an $\mb F[U]$-module, but \emph{not} as a graded module (see~\cite[Corollary 4.2]{Krcatovich}).

Krcatovich also observed that, if one is only concerned with correction terms, the connected sum of two L-space knots behaves as an L-space knot~\cite[Example 2]{Krcatovich}; more specifically, he showed that if $K$ and $K'$ are L-space knots, then $\rCFKm(K\#K')$ fits in a short exact sequence of complexes:
\[
0\to T \to \rCFKm(K\#K')\to A \to 0,
\]
where $T$ is a tower and $A$ is acyclic.
In this case, we will write $\rCFKm(K\#K')\approx T$; moreover, if $C$ is another chain complex such that $C\approx T$, we will also write $\rCFKm(K\#K')\approx C$.
In Krcatovich's terminology, $\rCFKm(K\#K')$ has a \emph{representative staircase}, which is determined by $T$;
conversely, the staircase determines $T$ and the collection $\{V_i(K\#K')\}$.
Moreover, for any other knot $L$, we can use $T$ as a substitute for $\rCFKm(K\#K')$ to compute $\CFKm(K\#K'\#L)$, in the sense that there is a filtered quasi-isomorphism
\[
T\otimes \CFKm(L) \cong \rCFKm(K\#K')\otimes \CFKm(L).
\]

\begin{proposition}
\label{prop:nonsoloilvinomigliorainvecchiando}
Let $K = T_{2,3} - T_{5,6}$. Then $\ourlim(K) = \frac{26}{5} < \ours(K) = 6$.
Moreover, $\ourlim(K) < \frac{\ours(nK)}{n}$ for all $n \in \Z_{>0}$, so the limit in Definition~\ref{def:ourlim}
is not attained at any $n$.
\end{proposition}


Before proving the proposition, recall that it is proven in~\cite{Bodnar} that, in the case of torus knots $T_{p,q}$, the representative staircase is determined by the arithmetics of $p$ and $q$ (compare also with~\cite[Section 5]{BorodzikLivingston}).
In what follows, we will be concerned with the connected sum $nT_{5,6}$ of $n$ copies of $T_{5,6}$, and in this case the result reads:
\begin{eqnarray*}
\rCFKm(n T_{5,6}) \approx \rCFKm(T_{5,5n+1}).
\end{eqnarray*}
That is, the representative staircase for $nT_{5,6}$ is the staircase of $T_{5,5n+1}$.

We will also need a lemma about $nT_{2,3}$. This is true in wider generality (see~\cite{Bodnar}), but we prove it here in a special case.
\begin{lemma}\label{lemma:staircase2n}
For each positive integer $n$, the complex $\CFKi(\pm nT_{2,3})$ is filtered chain homotopy equivalent to 
$\CFKi(\pm T_{2,2n+1})\oplus A_{\pm n}$, where $A_{\pm n}$ is an acyclic complex over $\F[U]$.
\end{lemma}

\begin{proof}
%
It suffices to prove the statement for $\CFKi(nT_{2,3})$, since the corresponding statement for $\CFKi(-nT_{2,3})$ follows by taking duals: in fact, $\CFKi(\mK)$ is isomorphic to the dual of $\CFKi(K)$, and taking duals preserves direct sums and acyclicity.

We will now prove the statement for $\CFKi(nT_{2,3})$ by induction on $n$: recall that $\CFKi((n+1)T_{2,3})$ is filtered quasi-isomorphic to $\CFKi(nT_{2,3})\otimes \CFKi(T_{2,3})$, and that $\CFKi(T_{2,3})$ is filtered quasi-isomorphic to $(\F[U,U^{-1}]a\oplus\F[U,U^{-1}]b\oplus\F[U,U^{-1}]c, \partial_1)$, where $\partial_1 b = Ua+c$ and $a$ and $c$ are cycles; moreover, the Alexander gradings of the generators are $A(a) = 1, A(b) = 0, A(c) = -1$.

By induction, we can assume that $\CFKi(nT_{2,3}) = \CFKi(T_{2,2n+1})\oplus A_n$, where $\CFKi(T_{2,2n+1})$ is generated over $\F[U,U^{-1}]$ by $x_1,\dots,x_{2n+1}$, is equipped with the differential $\partial_n$ defined by
\[
\partial_n x_{2i} = Ux_{2i-1} + x_{2i+1}, \quad \partial_n x_{2i+1} = 0,
\]
and the Alexander grading is $A(x_i) = n+1-i$.

We observe that, whenever $A$ is acyclic, $A\otimes C$ is acyclic for every other complex $C$. 
Therefore, in order to prove the theorem, it suffices to show that $\CFKi(T_{2,2n+1})\otimes\CFKi(T_{2,3})\cong \CFKi(T_{2,2n+3})\oplus A$, where $A$ is acyclic.

To this end, consider the subspace $V$ of $\CFKi(T_{2,2n+1})\otimes\CFKi(T_{2,3})$ spanned by:
\[
V = {\rm Span}_{\F[U,U^{-1}]}\left\{x_1a, x_1b, x_ic\right\},
\]
where we drop the $\otimes$ between generators to ease readability, so that $x_1a$ really means $x_1\otimes a$.
It is easy to check that $V$ is in fact a subcomplex of $\CFKi(T_{2,2n+1})\otimes\CFKi(T_{2,3})$, and that $V$ is indeed isomorphic to $\CFKi(T_{2,2n+3})$.
In fact, an explicit isomorphism is given by $x_1a\mapsto x_1, x_1b \mapsto x_2, x_ic\mapsto x_{i+2}$.

We claim that $V$ has a complement, which is the direct sum of copies of rank-4 subspaces $W_{2i}$, for $i=1,\dots, n$.
\[
W_{2i} = {\rm Span}_{\F[U,U^{-1}]}\left\{x_{2i}b, x_{2i-1}b + x_{2i}a, x_{2i+1}b+x_{2i}c,x_{2i+1}c\right\}.
\]
It is easy to prove that $W_{2i}$ is in fact an acyclic subcomplex for each $i$, and that the $W_{2i}$ together with $V$ span all of $\CFKi(T_{2,2n+1})\otimes\CFKi(T_{2,3})$.

Moreover, since the ranks of $V$ and $W_{2i}$ add up to the rank of $\CFKi(T_{2,2n+1})\otimes\CFKi(T_{2,3})$, this is actually a direct sum decomposition of complexes.
Since the $W_{2i}$ are acyclic, we have exhibited the desired decomposition.
\end{proof}

We can now turn to the proof of Proposition~\ref{prop:nonsoloilvinomigliorainvecchiando}.

\begin{proof}[Proof of Proposition~\ref{prop:nonsoloilvinomigliorainvecchiando}]
Let $K_1 = T_{2,3}$ and $K_2 = T_{5,6}$, $K = K_1 - K_2$.
The fact that $\ours(K) = 6$ was already observed in Example~\ref{ex:T25T56}. 
Let now $L_n = nK = nK_1 - nK_2$, and $n = 5\ell$. We will prove that for $\ell \in \Z_{>0}$ we have
\[
\ours(L_{5\ell}) = 26\ell+1.
\]
This implies at once that $\ourlim(K) = \lim_n \frac{\ours(L_n)}{n} = \frac{26}{5}$, and that $\ours(L_{5\ell}) > \ourlim(K)\cdot 5\ell$ for each $\ell$.
Moreover, by definition, for each $n$
\[
\ours(L_n) \ge \frac{26}{5} n
\]
for all $n \in \Z_{>0}$; since right-hand side is an integer only if $n$ is a multiple of $5$, the inequality is strict also for all $n$ not divisible by 5, hence the limit is never attained.


We now set out to prove that $\ours(L_{5\ell}) = 26\ell+1$.

Since $\rCFKm(nK_2)\approx \rCFKm(T_{5,5n+1})$, we can use Lemma~\ref{lemma:staircase2n} and results from~\cite{BCG2} to compute the invariants $V_i(nK_2 - nK_1)$, treating $nK_2$ as $T_{5,5n+1}$ and $-nK_1$ as $-T_{2,2n+1}$.
Indeed, let $J_i = 5\ell K_i$ for $i=1,2$.

Given a semigroup $\G \subseteq \N = \{0,1,\dots\}$, we denote by $\G(\cdot)$ its \emph{enumerating}
\emph{function}, i.e. the unique strictly increasing function
\[
\G \colon \N \to \N
\]
which is surjective on $\G$. Note that $\G(0) = 0$.
Given an integer $x$, we denote $(x)_+ = \max\set{0,x}$.
Since $\CFKi(-nT_{2,3})$ is, up to an acyclic summand, $\CFKi(-T_{2,2n+1})$, we can apply~\cite[Theorem 3.1 and Remark 3.3]{BCG2} and obtain:
\[
\ni^+_v(5\ell\mK) := \min\set{i \,\middle|\, V_i(5\ell\mK) \le v} = \left(\max_{k\geq0}\set{g(J_2) - g(J_1) + \G_{J_1}(k) - \G_{J_2}(k+v)}\right)_+,
\]
where $\G_{J_1}(\cdot)$ and $\G_{J_2}(\cdot)$ are the enumerating functions associated
to the semigroups
\[
\G_{J_1} = \spn{2, 10\ell+1}; \qquad \G_{J_2} = \spn{5, 25\ell+1}.
\]
The genera of the knots $J_1$ and $J_2$ are respectively $5\ell$ and $50\ell$,
so the formula for $\ni^+_v$ becomes
\begin{equation}
\label{e:ni+v}
\ni^+_v(\ol{L_{5\ell}}) = \left(45 \ell - \min_{k\geq0}\set{\G_{J_2}(k+v) - \G_{J_1}(k)}\right)_+.
\end{equation}
Note that, with this notation, we have that
\begin{equation}
\label{e:mistero}
\ours(L_{5\ell}) = \min_{v\geq0}\set{\ni^+_v(\ol{L_{5\ell}}) + 2v},
\end{equation}
which we are now going to compute.

The enumerating functions above can be expressed in the following equations:
\[
\G_{J_1}(k) =
\begin{cases}
2k & 0 \leq k \leq 5\ell\\
5\ell + k & k \geq 5\ell\\
\end{cases}
\]
\[
\G_{J_2}(k) =
\begin{cases}
5k & 0 \leq k \leq 5\ell\\
25\ell + 5\floor*{\frac{k-5\ell}{2}} + [k-5\ell]_2 & 5\ell \leq k \leq 15\ell\\
50\ell + 5\floor*{\frac{k-15\ell}{3}} + [k-15\ell]_3 & 15\ell \leq k \leq 30\ell\\
75\ell + 5\floor*{\frac{k-30\ell}{4}} + [k-30\ell]_4 & 30\ell \leq k \leq 50\ell\\
50\ell + k & k \geq 50\ell\\
\end{cases}
\]

Note that in Equation~\eqref{e:ni+v} we can in fact take the minimum over $0 \leq k \leq 5\ell$,
because for $k \geq 5\ell$ the function $\G_{J_1}(k)$ increases at a lesser or equal rate
than any translate of $\G_{J_2}$: specifically, $\G_{J_1}(k+j) - \G_{J_1}(k) = j \leq \G_{J_2}(k+v+j) - \G_{J_2}(k+v)$.
Therefore
\begin{equation*}
\ni^+_v(\ol{L_{5\ell}}) = \left(45 \ell - \min_{0 \leq k \leq 5\ell}\set{\G_{J_2}(k+v) - \G_{J_1}(k)}\right)_+.
\end{equation*}

Now we return to the proof of Proposition~\ref{prop:nonsoloilvinomigliorainvecchiando}.
Recall that we want to prove that $\ours(5\ell K) = 26\ell+1$. By~\eqref{e:mistero}
we have
\[
\ours(L_{5\ell}) = \min_{v \geq 0} \set{\ni^+_v(\ol{L_{5\ell}}) + 2v}.
\]
As shown in Lemma~\ref{lem:2} below, the choice $v = 13\ell$ gives $\ni^+_v(\ol{L_{5\ell}}) +2v = 26\ell+1$.
Moreover, it also follows from Lemma~\ref{lem:2} that $V_0(\ol{L_{5\ell}}) = 13\ell+1$, hence choosing $v\ge 13\ell+1$ yields $2v \ge 26\ell + 2 > 26\ell+1$.

We now distinguish between $v\le 5\ell-1$ and $v\ge 5\ell$.
By Lemma~\ref{lem:1} below, for $v \in [0, 5\ell-1]$ we have
\[
\ni^+_v(\ol{L_{5\ell}}) + 2v = 45\ell - 3v \geq 45\ell - 15\ell +3 > 26\ell +1;
\]
by Lemma~\ref{lem:3}, on the other hand, for $v \in [5\ell, 13\ell-1]$ we have
\[
\ni^+_v(\ol{L_{5\ell}}) + 2v \geq 2(13\ell-v) + 1 + 2v = 26\ell +1.
\]
This shows that $\ours(L_{5\ell}) = 26\ell + 1$, as desired.
\end{proof}

\begin{lemma}
\label{lem:2}
$\ni^+_{13\ell}(\ol{L_{5\ell}}) = 1$.
\end{lemma}
\begin{proof}
Note that, since $k \leq 5\ell$, $k+13\ell \in [13\ell, 18\ell]$.
Therefore, the difference of the enumerating functions is
\[
f(k) := \G_{J_2}(k+13\ell) - \G_{J_1}(k) =
\begin{cases}
45\ell + 5\floor*{\frac{k}{2}} + [k]_2 -2k & 0 \leq k \leq 2\ell\\
50\ell + 5\floor*{\frac{k-2\ell}{3}} + [k-2\ell]_3 -2k & 2\ell \leq k \leq 5\ell\\
\end{cases}
\]
In the first interval $f(k+2) \geq f(k)$, while in the second interval $f(k+3) \leq f(k)$.
It follows that the minimum is attained for some $k \in \set{0, 1, 5\ell -2, 5\ell -1, 5\ell}$.
A direct computation for these five values shows that the minimum is $45\ell-1$, attained both at $k=1$ and at $k=5\ell-1$.
It follows that $\ni^+_{13\ell}(\ol{L_{5\ell}}) = 45\ell - (45\ell-1) = 1$.
\end{proof}

\begin{lemma}
\label{lem:1}
For each $v = 0,\dots, 5\ell-1$, $\ni^+_v(\ol{L_{5\ell}}) = 45\ell - 5v$.
\end{lemma}

\begin{proof}
Note that, since we only need to test $k \leq 5\ell$ when computing the minimum in~\eqref{e:ni+v}, we can assume that for each value of $v$ in the statement $k+v \leq 10\ell-1$.
Therefore, the difference of the enumerating functions is
\[
f(k) := \G_{J_2}(k+v) - \G_{J_1}(k) =
\begin{cases}
5v+3k & 0 \leq k \leq 5\ell - v\\
25\ell + 5\floor*{\frac{k+v-5\ell}{2}} + [k+v-5\ell]_2 -2k & 5\ell - v \leq k \leq 5\ell.\\
\end{cases}
\]
Such a function is increasing on the interval $0 \le k \le  5\ell-v$,
and on the second interval it satisfies the condition $f(k+2) - f(k) \geq 1$.
It follows that the minimum is attained for some $k = 0, 5\ell-v$ or $5\ell-v +1$.
A direct computation for these values shows that the minimum is $5v$, attained at $k=0$.
Therefore, $\nu^+_v(\ol{L_{5\ell}}) = 45\ell-5v$.
\end{proof}

\begin{lemma}
\label{lem:3}
Let $v = 13\ell -s$ for some $0 < s \leq 8\ell$. Then $\ni^+_v(\ol{L_{5\ell}}) \geq 2s+1$.
\end{lemma}
\begin{proof}
Choosing $k=0$ in Equation~\eqref{e:ni+v}, we obtain:
\[
\ni^+_v(\ol{L_{5\ell}}) \geq 45 \ell - \G_{J_2}(13\ell-s).
\]
Since $13\ell -s \in [5\ell, 13\ell] \subseteq [5\ell, 15\ell]$, we have
\[
\G_{J_2}(13\ell-s) = 45\ell + 5\floor*{-\frac{s}{2}} + [s]_2.
\]
If $s \geq 2$ is even, then $\G_{J_2}(13\ell-s) = 45 \ell - \frac{5}{2}s \leq 45\ell -2s -1$.
If $s$ is odd, then $\G_{J_2}(13\ell-s) = 45 \ell - \frac{5}{2}(s+1) + 1 \leq 45\ell -2s -1$.
In both cases we have $\G_{J_2}(13\ell-s) \leq 45\ell -2s -1$, so we obtain
\[
\ni^+_v(\ol{L_{5\ell}}) \geq 45 \ell - \G_{J_2}(13\ell-s) \geq 2s + 1.
\qedhere
\]
\end{proof}

With techniques similar to the ones used in Proposition~\ref{prop:nonsoloilvinomigliorainvecchiando}
one can to show that $\ourlim$ attains many other positive non-integer values.
We conclude with two questions concerning the image of $\ourlim$.

\begin{question}
Is $\Q_{\geq0} \subseteq \im(\ourlim)$?
Can $\ourlim$ take irrational values?
\end{question}

\bibliographystyle{amsplain}
\bibliography{topology}
\end{document}